\documentclass[12pt]{article}

\usepackage{amsfonts}
\usepackage{amssymb}
\usepackage{amsthm}
\usepackage{amsmath}
\usepackage{a4wide}
\usepackage{mathrsfs}
\usepackage{epsfig}
\usepackage{palatino}

\usepackage{graphicx}

\usepackage{color}
\definecolor{citegreen}{rgb}{0,0.6,0}
\definecolor{refred}{rgb}{0.8,0,0}
\usepackage[colorlinks, citecolor=citegreen,linkcolor=refred]{hyperref}

\usepackage{mathtools}
\mathtoolsset{showonlyrefs}

\usepackage[]{draftwatermark}
\SetWatermarkText{${\mathcal{DRAFT}}$}
\SetWatermarkScale{6.2}

\newtheorem{thm}{Theorem}[section]

\newtheorem{prop}[thm]{Proposition}

\theoremstyle{definition}
\newtheorem{defn}[thm]{Definition}

\theoremstyle{remark}
\newtheorem{rem}[thm]{Remark}

\numberwithin{equation}{section}

\def\Ric{{\mathrm {Ric}}}

\def\R{\mathbb R}

\def\R{{{\mathbb R}}}

\newcommand{\intbar}{\etaathop{\int\etaakebox(-13.5,0){\rule[4pt]{.7em}{0.3pt}}
\kern-6pt}\nolimits}

\newcommand{\be}{\begin{equation}}
\newcommand{\ee}{\end{equation}}
\newcommand{\bea}{\begin{equation*}}
\newcommand{\eea}{\end{equation*}}

\usepackage{tikz,fp,ifthen}
\usetikzlibrary{backgrounds}
\usetikzlibrary{decorations.pathmorphing,backgrounds,fit,calc,through}
\usetikzlibrary{arrows}
\usetikzlibrary{shapes,decorations,shadows}
\usetikzlibrary{fadings}
\usetikzlibrary{patterns}
\usetikzlibrary{mindmap}
\usetikzlibrary{decorations.text}
\usetikzlibrary{decorations.shapes}

\title{A Liouville theorem for superlinear heat equations on Riemannian manifolds}

\author{Daniele Castorina \footnote{Dipartimento di Matematica e Applicazioni, Universit\`a di Napoli, Via Cintia, Monte S. Angelo 80126 Napoli, Italy, daniele.castorina@unina.it} \and Carlo Mantegazza \footnote{Dipartimento di Matematica e Applicazioni, Universit\`a di Napoli, Via Cintia, Monte S. Angelo 80126 Napoli, Italy, c.mantegazza@sns.it} \and Berardino Sciunzi \footnote{Dipartimento di Matematica e Informatica, Universit\`a della Calabria, V. P. Bucci, Arcavacata di Rende (CS), Italy, sciunzi@mat.unical.it.}}

\begin{document}

 \maketitle
\begin{abstract} We study the triviality of the solutions of weighted superlinear heat equations on Riemannian manifolds with nonnegative Ricci tensor. We prove a Liouville--type theorem for solutions bounded from below with nonnegative initial data, under an integral growth condition on the weight.
\end{abstract}

\section{Introduction}

The purpose of this note is the study of the triviality (constancy in space and time) of the solutions of weighted superlinear heat equations on Riemannian manifolds with nonnegative Ricci tensor. We will prove a Liouville--type theorem for solutions which are bounded from below with nonnegative initial data, under an integral growth condition on the weight. Precisely, we investigate the nonexistence of classical ($C^2$ in space and $C^1$ in time) nontrivial solutions $u \in C^{2,1} (M \times (0,+\infty)) \cap C^0 (M \times [0,+\infty))$ of the problem
\begin{equation}\label{eq-1}
u_t=\Delta u+ |u|^{p} V \quad \text{ in } M \times (0,+\infty) \quad \text{ and } \quad u (\cdot,0) = u_0\geq 0,
\end{equation}
where $(M,g)$ is a smooth, connected and complete Riemannian manifold of dimension $n$, $V$ is a given positive function, $\Delta$ denotes the Laplace--Beltrami operator on $(M,g)$ and $p>1$. Concerning the positive weight $V$, we will assume that $V \in C^0 (M\times (0,+\infty))$.

In the last decades, finite time blow--up and global existence of solutions to semilinear parabolic equations like the one of problem~\eqref{eq-1} in general domains of $\R^n$ and various time intervals $I\subset\R$ have attracted quite a lot of interest in the literature. When $M=\R^n$ and $V \equiv 1$, that is, for the simpler equation $u_t=\Delta u+|u|^{p}$, in the pioneering paper~\cite{Fujita}, Fujita proved the nonexistence of nonzero nonnegative solutions on $M\times[0+\infty)$, for exponents $p>1$ below the optimal threshold $p_F=1+2/n$ (see Remark~\ref{sharp}). Subsequently, several generalizations to more quasilinear parabolic operators ($p$--Laplacian or porous medium--type, for instance) have been developed, see~\cite{galak2, galak3, mitipoho,pohotesei} and references therein. At the same time, extensions have been carried out also in the context of Riemannian manifolds: in the papers~\cite{qszhang1, qszhang2} (see also~\cite{mastrolia2018} and references therein), Zhang proved the same result of Fujita for the equation $u_t=\Delta u+ |u|^{p} V$ on a Riemannian manifold $(M,g)$, under some geometrical assumptions and growth conditions on the weight $V$ in the case this latter is time--independent. In this work we remove the positivity hypothesis by such conclusion of nonexistence of nonzero solutions (equivalently, {\em of triviality}), assuming only a bound from below that, as we will see, implies the positivity of the solutions, moreover, we extend the work of Zhang by considering weights $V$ which can depend also on time. 

These Liouville--type results are special cases of a general research topic about finding conditions under which the solutions of the equation $u_t=\Delta u+ |u|^{p} V$ on $M\times I$, where $I\in\R$ is an unbounded interval, are actually zero or depending only on time (when $V$ depends only on time), hence given by the (possibly explicit) 1--dimensional profile obtained solving the ODE $u_t= |u|^{p} V$. More precisely, when the interval $I$ is $[T,+\infty)$, $(-\infty,T)$, for some $T\in\R$, or the whole $\R$, we speak of {\em immortal}, {\em ancient} and {\em eternal} solutions, respectively. Then, these theorems can be seen as ``classification results'' for the solutions, the immortal case being the one considered in this paper and in the above mentioned literature. Clearly, ancient and immortal solutions are special cases of the eternal ones, hence our conclusions apply also to these latter, even if it must be said that being eternal is a far more restrictive assumption than being only immortal, indeed stronger results are available (see~\cite{qsbook}, for instance). The interest in these special solutions is given by the fact that they arise as blow--up limits when the solutions of semilinear parabolic equations develop a singularity, that is, for instance, a solution in a {\em bounded} time interval $(0,T)$ which becomes unbounded as $t \to T^-$. 

In the case $V \equiv 1$, the analysis of the triviality of ancient and eternal solutions on Riemannian manifolds has been recently partially addressed by the first two authors in~\cite{cama1} and~\cite{cama2}. Anyway, even in such simplest case, some results holding in the Euclidean space still do not possess an analogue when the ambient is a Riemannian manifold, under the (expected) optimal geometric assumptions. For instance (up to our knowledge), we mention the expected constancy in space of the positive ancient solutions of the equations $u_t=\Delta u+ |u|^{p}$ on an $n$--dimensional Riemannian manifold with nonnegative Ricci tensor, for every $p>1$ if $n=2$ and for $1<p< \frac{n(n+2)}{(n-1)^2}$ when $n\geq 3$ (see Giga and Kohn~\cite{gigakohn}, Bidaut--V\'eron~\cite{bveron}, Merle and Zaag~\cite{merlezaag}, Polacik, Quittner and Souplet~\cite{pqs2}, for the case $\R^n$). We refer the interested reader to the papers of Souplet and Zhang~\cite{souzha}, Polacik, Quittner and Souplet~\cite{pqs2} and to the book of Quittner and Souplet~\cite{qsbook}, for the state--of--the--art in the Euclidean space.

Besides trying to extend these results to Riemannian manifolds, the general case of ancient/eternal solutions of the equation $u_t=\Delta u+ |u|^{p} V$ with a nonconstant positive weight $V$, will be subject of future investigation (analogously to what we are doing in this paper), in particular in the special case of a weight $V$ depending only on time.\\

Throughout all the paper, we will assume the nonnegativity of the Ricci tensor $\Ric$ of the manifold $(M,g)$, with the meaning that all its eigenvalues are nonnegative. By the Bishop--Gromov inequality (see Section 11.1.3 in~\cite{petersen1}), this implies that there exists a constant $C>0$ such that
\begin{equation}\label{Ric1}
\mu(B_R) \leq C R^n \; \text{ for any } R>0, 
\end{equation}
where $\mu$ is the canonical volume measure of $(M,g)$, $n$ is the dimension of the manifold and $B_R$ denotes any geodesic ball in $M$ with radius $R$.
	
\begin{defn}
We say that the conditions {\bf (VOL)} hold for $V$ if there exist a point $x_0\in M$ and positive
constants $C$, $R_0$ and $\alpha$ such that for every $R\geq R_0$, we have
\begin{equation}\label{VOL1}
\int_{R^{2}/4}^{R^{2}} \int_{B_R} V(x,t)^{-\frac{1}{p-1}} \,d\mu \,dt \leq 
CR^{2+\alpha}
\end{equation}
and
\begin{equation}\label{VOL2}
\int_{0}^{R^2} \int_{B_R \setminus B_{R/2}} V(x,t)^{-\frac{1}{p-1}} \,d\mu \,dt \leq
CR^{2+\alpha},
\end{equation}
where $B_R=B_R(x_0)$ is the geodesic ball centered in $x_0$ with radius $R$.
\end{defn}

Our main result is the following theorem.

\begin{thm}\label{main}
Let $(M,g)$ be a complete $n$--dimensional Riemannian manifold with $\Ric \geq 0$. Let $u \in C^{2,1}(M\times(0,+\infty))\cap C^{0}(M\times [0,+\infty))$ be a classical solution of the equation $u_t = \Delta u + |u|^p V$ in $M \times (0,+\infty)$, bounded from below and such that $u (\cdot, 0) = u_0 \geq 0$, with $V \in C^{0}(M\times (0,+\infty))$ positive and satisfying conditions {\bf (VOL)}, with $1 < p \leq 1 + {2}/{\alpha}$. Then $u \equiv 0$.
\end{thm}

In the case $V \equiv 1$, the estimate~\eqref{Ric1} easily implies that the conditions {\bf (VOL)} are satisfied for $\alpha = n$. In particular, as a corollary of Theorem~\ref{main}, we recover the triviality of the positive solutions of problem~\eqref{eq-1} up to the critical Fujita exponent (i.e. for $p \leq 1 + {2}/{n}$), in accordance with~\cite{qszhang1} and~\cite{qszhang2}. Moreover, we underline that previous similar Liouville--type theorems (see for instance~\cite{mastrolia2018} and references therein) require the nonnegativity of the solutions, while here we just assume their boundedness from below.

\begin{rem}\label{sharp}
The assumptions of Theorem~\ref{main} about the nonnegativity of the initial datum $u_0$ as well as the restriction on the exponent $p$ are optimal.

If we consider the function
\begin{equation*}
v(x,t) = - \left( (p-1)(t+1) \right)^{-\frac{1}{p-1}},
\end{equation*}
we easily see that $v$ satisfies $v_t = \Delta v + |v|^p$ in $M \times (0,+\infty)$ for any $p>1$, with $v (x, 0) = - (p-1) ^{-\frac{1}{p-1}} <0$, so that $v$ is a negative, {\em nonconstant}, bounded below, global solution to problem~\eqref{eq-1} for $V\equiv1$. Thus, the hypothesis on the nonnegativity of the initial datum cannot be removed.

In the Euclidean case $M = \R^n$, choosing $V\equiv1$, so that $\alpha = n$ in the conditions {\bf (VOL)} above, Fujita in~\cite{Fujita} proved that if $p > 1+{2}/{n}$, problem~\eqref{eq-1} admits a {\em nontrivial} positive global solution on $M\times[0,+\infty)$ (we mention that the critical case $p=1+2/n$,  ``missing'' in the work of Fujita, was first addressed by Zhang in~\cite{qszhang2} showing the nonexistence of nonzero solutions). Hence, the above restriction on the exponent $p$ is necessary.
\end{rem}

We will prove Theorem~\ref{main} in the next two sections: in Section~\ref{positivity} we will see that the boundedness from below implies the nonnegativity of classical ``supercaloric'' functions (i.e. with nonnegative heat operator) with nonnegative initial data, hence also of the solutions of problem~\eqref{eq-1}; in Section~\ref{triviality} we will show that, under the hypotheses of Theorem~\ref{main}, any nonnegative solution is constant.

\bigskip

\section{Positivity}\label{positivity}

In this section, thanks to an argument by Ecker and Huisken in~\cite{eckhui2}, we prove that any classical ``supercaloric'' function over $M \times [0,+\infty)$, with $\Ric \geq 0$, which is bounded from below and nonnegative at $t=0$, is globally nonnegative. This will clearly yield, as a particular case, that any bounded below solution of problem~\eqref{eq-1}, with $V>0$ and $u_0 \geq 0$, is nonnegative.\\

\begin{prop}\label{proppos}
Let $(M,g)$ be a complete $n$--dimensional Riemannian manifold with $\Ric \geq 0$. Let $u \in C^{2,1} (M \times (0,+\infty)) \cap C^0 (M \times [0,+\infty))$ be bounded from below and satisfies $u_t - \Delta u \geq 0$ in $M \times (0,+\infty)$ with $u (\cdot, 0) = u_0 \geq 0$, then $u \geq 0$.
\end{prop}
\begin{proof}
As $u$ is bounded from below, we have $u \geq - A/2$ for some constant $A>0$. We consider the function $v = - u$, which satisfies $v \leq A/2$ and
\begin{equation}\label{eq-u}
v_t - \Delta v \leq 0 \; \text{ in } \; M \times (0,+\infty).
\end{equation}
Let us define
\begin{equation*}
F(s) = \frac{1}{A-s}
\end{equation*}
and notice that the function $F$ is positive, increasing, convex and satisfies
\begin{equation}\label{huisk0}
F''(s) F(s) = 2 (F'(s))^2,
\end{equation}
in the interval $(-\infty,A)$. From inequality~\eqref{eq-u} we then have
\begin{equation}\label{eqF}
\partial_t F(v) - \Delta F(v) = F'(v) v_t - F'(v) \Delta v - F''(v) |\nabla v|^2 \leq - F''(v) |\nabla v|^2 \leq 0.
\end{equation}
For any $x_0\in M$, letting $r=d(x,x_0)$ where $d$ is the geodesic distance in $(M,g)$, we consider the function
\begin{equation*}
\phi(x,t) = (R^2 - r^2 - 2nt)_+
\end{equation*}
where the sign $+$ means {\em positive part}. By the {\em Laplacian comparison theorem} (see~\cite[Chapter~9, Section~3.3]{petersen1}), the condition $\Ric\geq0$ implies 
\begin{equation}\label{lapcom}
\Delta r \leq \frac{n-1}{r},
\end{equation}
{\em in the sense of support functions} (or {\em in the sense of viscosity}, see~\cite{crisli1} -- check also~\cite[Appendix~A]{manmasura} for comparison of the two notions), in particular, this inequality can be used in maximum principle arguments (see again~\cite[Chapter~9, Section~3]{petersen1}, for instance). 
Recalling that $|\nabla r|=1$, we have then
\begin{equation}\label{eqphi}
\phi_t - \Delta \phi = \Delta r^2 -2n = 2r \Delta r + 2 |\nabla r|^2 - 2n \leq 2 (n-1) + 2 -2n = 0
\end{equation}
in the set $\{\phi>0\}$.\\
Let us set $w(x,t)= \phi(x,t) F(v(x,t))$ which is a well defined function $w:M\times[0,+\infty)\to\R$, since $v < A$, positive and smooth in 
$\{w>0 \} \cap (M \times (0,+\infty))=\{\phi>0 \} \cap (M \times (0,+\infty))$.\\
Hence, thanks to formulas~\eqref{eqF} and~\eqref{eqphi}, we see that $w$ satisfies
\begin{equation}\label{huisk1}
\begin{split}
\partial_t w - \Delta w &= (\partial_t - \Delta) (\phi F(v))\\
&= \phi (\partial_t - \Delta) F(v) + F(v)(\partial_t - \Delta) \phi - 2 \nabla \phi \nabla F(v)\\ 
&\leq - \phi F''(v) |\nabla v|^2 - 2 \nabla \phi F'(v) \nabla v
\end{split}
\end{equation}
in $\{w>0\} \cap (M \times (0,+\infty))$.\\
Then, for any fixed $t>0$, the function $w(\cdot, t)$ has compact support and at any maximum point $x_t$ of $w(\cdot,t)$ with $w(x_t,t)>0$, we have
\begin{equation*}
0 = \nabla w (x_t,t) = F(v(x_t,t)) \nabla \phi (x_t,t) + \phi (x_t,t) F'(v(x_t,t)) \nabla v(x_t,t),
\end{equation*}
that is,
\begin{equation*}
\nabla \phi (x_t,t) = -\phi (x_t,t) \frac{F'(v(x_t,t))}{F(v(x_t,t))} \nabla v (x_t,t).
\end{equation*}
Thus, from formulas~\eqref{huisk0} and~\eqref{huisk1}, we conclude
\begin{equation}\label{huisk2}
\partial_t w - \Delta w \leq \phi |\nabla v|^2 \left[ \frac{2 (F'(v))^2}{F(v)} - F''(v) \right] = 0
\end{equation}
at a point $(x_t,t) \in \{w>0\} \cap (M \times (0,+\infty))$ as above.\\
Setting $w_{\mathrm{max}} (t) = \max_{x \in M} w(x,t)$, it follows by Hamilton's trick (see~\cite{hamilton2}) that for almost every $t \in (0, R^2/2n )$ there holds $w_{\max}' (t) \leq 0$, which integrated implies
\begin{equation}\label{huisk3}
w(x,t) \leq \max_{x\in M} w(x,0)
\end{equation}
for every $x\in M$ and $t\geq 0$.\\
Then, for any $0<\delta<1$, in the set 
\begin{equation*}
B_\delta = \bigl\{(x,t) \in M \times (0,+\infty)\,:\, R^2 \delta - r^2 - 2nt \geq 0\, \bigr\}
\end{equation*}
we have $\phi (x,t) \geq (1-\delta) R^2$ and
\begin{equation*}
\phi(x,t) F(v(x,t)) \leq \sup_{x\in M} F(v(x,0)) (R^2 - r^2) \leq \sup_{x\in M} F(v(x,0)) R^2,
\end{equation*}
hence,
\begin{equation*}
(1-\delta) R^2 F(v(x,t)) \leq\sup_{x\in M} F(v(x,0)) R^2.
\end{equation*}
This last inequality, from the definition of $F$, reads
\begin{equation*}
\frac{(1-\delta) R^2}{A-v(x,t)} \leq \frac{R^2}{A-{\sup_{x\in M}} v(x,0)},
\end{equation*}
that is,
\begin{equation*}
v(x,t) \leq (1-\delta) \sup_{x\in M} v(x,0) + \delta A\leq \delta A,
\end{equation*}
as $v(\cdot ,0)=-u_0\leq 0$. Being $v=-u$, we conclude
\begin{equation}\label{huisk}
u(x,t) \geq  -\delta A
\end{equation}
for any $(x,t) \in B_\delta$. Sending $R\to+\infty$, the set $B_\delta$ becomes $M\times(0,+\infty)$, hence the above inequality holds everywhere. Sending then $\delta\to0^+$, we get $u \geq0$ on $M\times[0,+\infty)$, as claimed.
\end{proof}

\begin{rem}
This argument clearly gives another proof of the fact that a bounded solution of the heat equation on manifolds with nonnegative Ricci curvature is unique (see for instance Theorem 5.1 in~\cite{liyau}).
\end{rem}

\bigskip

\section{Triviality}\label{triviality}

By Proposition~\ref{proppos} in the previous section, we know that, under the hypotheses of Theorem~\ref{main}, any classical solution of problem~\eqref{eq-1} bounded from below and with $u(\cdot,0) = u_0 \geq 0$ is nonnegative. In this section, we will prove that any nonnegative solution is constant, hence identically zero, following~\cite{qszhang1,qszhang2}.\\

\noindent From now on we will let $C$ be a positive constant which may change, from time to time, even within the same line. We introduce two cut--off functions $\phi$ and $\eta$ with the following properties:
\begin{enumerate}
\item $0 \leq \phi \leq 1$; $\phi(r) = 1$ for $0 \leq r \leq 1/2$; $\phi(r) = 0$ for $r \geq 1$; $-C \leq \phi'(r) \leq 0$; $|\phi''(r)| \leq C$; 
\item $0 \leq \eta \leq 1$; $\eta(t) = 1$ for $0 \leq r \leq 1/4$; $\eta(t) = 0$ for $t \geq 1$; $-C \leq \eta'(r) \leq 0$.
\end{enumerate}
Let $x_0\in M$ be the point appearing in conditions \textbf{(VOL)}. For any $R>0$ let us set $Q_R = B_R \times (0,R^2]$ (with $B_R = B_R(x_0)$ being the geodesic ball of radius $R$ and center $x_0$), and consider the cut--off function $\psi_R (x,t) = \phi_R (r) \eta_R (t)$, where $r = d (x,x_0)$, $\phi_R (r) = \phi (r/R)$ and $\eta_R (t) = \eta (t/R^2)$, with $\phi$ and $\eta$ as above. Clearly, we have
\begin{equation}\label{hypcutoff}
-\frac{C}{R} \leq \phi'_R \leq 0, \ \ \ \ \ \ \left| \phi''_R\right| \leq \frac{C}{R^2}\ \ \ \ \ \ \text{ and }\ \ \ \ -\frac{C}{R^2} \leq \eta'_R\leq 0.
\end{equation}
Let us consider 
\begin{equation*}
I_R = \int_{Q_R} V(x,t) u^{p} (x,t) \psi_{R}^{q} (x,t) \,d\mu \,dt,
\end{equation*}
with $q = p'$, from the superlinear heat equation we then have
\begin{equation*}
I_R = \int_{Q_R} [u_t (x,t) - \Delta u(x,t) ] \psi_{R}^{q} (x,t) \,d\mu \,dt = \int_{0}^{R^2}\int_{B_R} [u_t (x,t) - \Delta u(x,t) ] \psi_{R}^{q} (x,t) \,d\mu \,dt.
\end{equation*}
Since $\psi_R$ is Lipschitz and $\psi_R (\cdot,t)= 0$ on $\partial B_R$ for every $t \in [0,R^2]$, $\nabla \psi_R = 0$ on $\partial B_R \times [0,R^2]$, we get
\begin{equation*}
I_R = \int_{B_R} \int_{0}^{R^2} u_{t} (x,t) \psi_{R}^{q} (x,t) \,dt d\mu + \int_{0}^{R^2} \int_{B_R} \nabla u(x,t) \nabla \psi_{R}^{q} (x,t) \,d\mu \,dt.
\end{equation*}
Hence, integrating by parts and recalling the definition of $\psi_R$, we see that
\begin{equation}\label{zhang1} 
\begin{split} 
I_R &=\int_{B_R} u (x,\cdot) \psi_{R}^{q} (x,\cdot)\big\vert_{0}^{R^2} \,d\mu - \int_{Q_R} u(x,t) \phi_{R}^{q} (x) q \eta_{R}^{q-1} (t) \eta'_{R} (t) \,d\mu \,dt.\\
 &+ \int_{0}^{R^2} \int_{\partial B_R} u (x,t) \nabla \phi_{R}^{q} (x) \cdot \nu\,\eta_{R}^{q} (t) \,d\sigma \,dt -\int_{0}^{R^2} \int_{B_R} u (x,t) \Delta \phi_{R}^{q} (x) \eta_{R}^{q} (t) \,d\mu \,dt,
\end{split}
\end{equation}
where $\nu$ are $\sigma$ are respectively the unit outward normal and the canonically induced measure on $\partial B_R$.
Since $u \geq 0$, $\psi_R(x,R^2)=0$ and as $\phi'_R(R) = 0$ implies
\begin{equation*}
\nabla \phi_{R}^{q} \cdot \nu = q \phi_{R}^{q-1} \phi'_{R} \nabla r \cdot \nu = 0 \, \text{ on } \, \partial B_R,
\end{equation*}
by equation~\eqref{zhang1}, we get
\begin{equation*}
I_R \leq -\int_{Q_R} u(x,t) \phi_{R}^{q} (x) q \eta_{R}^{q-1} (t) \eta'_{R} (t) \,d\mu \,dt -\int_{0}^{R^2} \int_{B_R} u (x,t) \Delta \phi_{R}^{q} (x) \eta_{R}^{q} (t) \,d\mu \,dt.
\end{equation*}
As 
\begin{equation*}
\Delta \phi_{R}^{q} = q \phi_{R}^{q-1} \Delta \phi_{R} + q (q-1) \phi_{R}^{q-2}|\nabla \phi_{R}|^2 \geq q \phi_{R}^{q-1} \Delta \phi_{R},
\end{equation*}
plugging it into the above inequality, we get
\begin{equation*}
I_R \leq -\int_{Q_R} u(x,t) \phi_{R}^{q} (x) q \eta_{R}^{q-1} (t) \eta'_{R} (t) \,d\mu \,dt -\int_{0}^{R^2} \int_{B_R} u (x,t) q (\phi_{R}^{q-1} \Delta \phi_{R}) (x) \eta_{R}^{q} (t) \,d\mu \,dt.
\end{equation*}
Taking into account the supports for $\phi_R$ and $\eta_R$ we then obtain,
\begin{equation*}
\begin{split}
I_R \leq&\, -\int_{R^{2}/4}^{R^{2}} \int_{B_R} u(x,t) \phi_{R}^{q} (x) q \eta_{R}^{q-1} (t) \eta'_{R} (t) \,d\mu \,dt\\
 &-\int_{0}^{R^2} \int_{B_R \setminus B_{R/2}} u (x,t) q (\phi_{R}^{q-1} \Delta \phi_{R}) (x) \eta_{R}^{q} (t)\,d\mu \,dt.
\end{split}
\end{equation*}
By properties~\eqref{hypcutoff} of $\phi_R$, recalling that $|\nabla r|=1$ and thanks to inequality~\eqref{lapcom}, we have
\begin{equation*}
\Delta \phi_{R} = |\nabla r|^2 \phi''_{R} + \Delta r \phi'_{R} \geq \phi''_{R} + \frac{n-1}{r} \phi'_{R} \geq - \frac{C}{R^2} \quad \text{ in } B_R \setminus B_{R/2},
\end{equation*}
which implies
\begin{equation*}
I_R \leq \frac{C}{R^2} \biggl[ \int_{R^{2}/4}^{R^{2}} \int_{B_R} u(x,t) \phi_{R}^{q} (x) \eta_{R}^{q-1} (t) \,d\mu \,dt +
\int_{0}^{R^2} \int_{B_R \setminus B_{R/2}} u (x,t) \phi_{R}^{q-1} (x) \eta_{R}^{q} (t)\,d\mu \,dt\biggr].
\end{equation*}
Moreover, since $\phi_R, \eta_R \leq 1$, there holds
\begin{equation}\label{zhang3}
I_R \leq \frac{C}{R^2} \biggl[\int_{R^{2}/4}^{R^{2}} \int_{B_R} u(x,t) \psi_{R}^{q-1} (x,t) \,d\mu \,dt + \int_{0}^{R^2} \int_{B_R \setminus B_{R/2}} u (x,t) \psi_{R}^{q-1} (x,t) \,d\mu \,dt\biggr].
\end{equation}
Now, applying H\"older inequality to both terms in the right hand side of this inequality, we obtain
\begin{equation*}
\begin{split}
&\int_{R^{2}/4}^{R^{2}} \int_{B_R} u(x,t) \psi_{R}^{q-1} (x,t) \,d\mu \,dt = \int_{R^{2}/4}^{R^{2}} \int_{B_R} V^{\frac{1}{p}} (x) u(x,t) \psi_{R}^{q-1} (x,t) V^{-\frac{1}{p}} (x) \,d\mu \,dt \\
& \leq \biggl( \int_{R^{2}/4}^{R^{2}} \int_{B_R} V(x,t) u^p (x,t) \psi_{R}^{q} (x,t) \,d\mu \,dt\biggr)^{\frac{1}{p}} \biggl( \int_{R^{2}/4}^{R^{2}} \int_{B_R} V(x,t)^{-\frac{1}{p-1}} \,d\mu \,dt\biggr)^{\frac{1}{q}},
\end{split}
\end{equation*}
and
\begin{equation*}
\begin{split}
&\int_{0}^{R^2} \int_{B_R \setminus B_{R/2}} u(x,t) \psi_{R}^{q-1} (x,t) \,d\mu \,dt = \int_{0}^{R^2} \int_{B_R \setminus B_{R/2}} V^{\frac{1}{p}} (x) u(x,t) \psi_{R}^{q-1} (x,t) V^{-\frac{1}{p}} (x)\,d\mu \,dt \\
&\leq \biggl( \int_{0}^{R^2} \int_{B_R \setminus B_{R/2}} V(x,t) u^p (x,t) \psi_{R}^{q} (x,t) \,d\mu \,dt\biggr)^{\frac{1}{p}} \biggl( \int_{0}^{R^2} \int_{B_R \setminus B_{R/2}} V(x,t)^{-\frac{1}{p-1}} \,d\mu \,dt\biggr)^{\frac{1}{q}},
\end{split}
\end{equation*}
which, substituted into inequality~\eqref{zhang3} give
\begin{equation}\label{zhang4}
\begin{split}
 I_R &\leq\,\frac{C}{R^2} \biggl(\int_{R^{2}/4}^{R^{2}} \int_{B_R} V(x,t) u^p (x,t) \psi_{R}^{q} (x,t) \,d\mu \,dt \biggr)^{\frac{1}{p}}\biggl( \int_{R^{2}/4}^{R^{2}} \int_{B_R} V(x,t)^{-\frac{1}{p-1}} \,d\mu \,dt\biggr)^{\frac{1}{q}} \\
&+ \frac{C}{R^2} \biggl( \int_{0}^{R^2} \int_{B_R \setminus B_{R/2}} V(x,t) u^p (x,t) \psi_{R}^{q} (x,t) \,d\mu \,dt \biggr)^{\frac{1}{p}}\biggl( \int_{0}^{R^2} \int_{B_R \setminus B_{R/2}} V(x,t)^{-\frac{1}{p-1}} \,d\mu \,dt\biggr)^{\frac{1}{q}}.
\end{split}
\end{equation}
This estimate then implies 
\begin{equation*}
I_R \leq I_{R}^{\frac{1}{p}} \frac{C}{R^2} \biggl[ \biggl( \int_{R^{2}/4}^{R^{2}} \int_{B_R} V(x,t)^{-\frac{1}{p-1}} \,d\mu \,dt\biggr)^{\frac{1}{q}} + \biggl( \int_{0}^{R^2} \int_{B_R \setminus B_{R/2}} V(x,t)^{-\frac{1}{p-1}} \,d\mu \,dt\biggr)^{\frac{1}{q}} \biggr],
\end{equation*}
that is,
\begin{equation}\label{zhangfinal0}
I_{R}^{1-\frac{1}{p}} \leq \frac{C}{R^2} \biggl[ \biggl( \int_{R^{2}/4}^{R^{2}} \int_{B_R} V(x,t)^{-\frac{1}{p-1}} \,d\mu \,dt\biggr)^{\frac{1}{q}} + \biggl( \int_{0}^{R^2} \int_{B_R \setminus B_{R/2}} V(x,t)^{-\frac{1}{p-1}} \,d\mu \,dt\biggr)^{\frac{1}{q}} \biggr].
\end{equation}
Now, since conditions \textbf{(VOL)} hold, substituting them into inequality~\eqref{zhangfinal0}, we obtain
\begin{equation*}
I_{R}^{\frac{1}{q}} = I_{R}^{1-\frac{1}{p}} \leq C R^{\frac{2+\alpha}{q}-2} = C\left(R^{2+\alpha-2q}\right)^{\frac{1}{q}},
\end{equation*}
that is,
\begin{equation}\label{zhangfinal}
I_{R} \leq C R^{2+\alpha-2q}.
\end{equation}
Finally, if $1 < p < 1 + 2/\alpha$, we have that $2+\alpha-2q < 0$, hence, if we take the limit as $R \to +\infty$, we conclude 
\begin{equation*}
\lim_{R \to +\infty} I_R = \int_{M\times(0,+\infty)} V(x,t) u^{p} (x,t) \,d\mu \,dt = 0,
\end{equation*}
which, since $V>0$, clearly implies $u = 0$, as claimed. Moreover, if $p=1 + 2/\alpha$, we have that $2+\alpha-2q = 0$, so that  
\begin{equation*}
\lim_{R \to +\infty} I_R = \int_{M\times(0,+\infty)} V(x,t) u^{p} (x,t) \,d\mu \,dt \leq C.
\end{equation*}
Thus, from inequality~\eqref{zhang4} we deduce
\begin{equation*}
I_R \leq\,C\biggl[\biggl(\int_{R^{2}/4}^{R^{2}} \int_{B_R} V(x,t) u^p (x,t) \,d\mu \,dt \biggr)^{\frac{1}{p}} \\
+ \biggl( \int_{0}^{R^2} \int_{B_R \setminus B_{R/2}} V(x,t) u^p (x,t) \,d\mu \,dt \biggr)^{\frac{1}{p}}\biggr],
\end{equation*}
whose RHS tends to zero, as $R \to +\infty$, and we can conclude again that $u = 0$.\qed

\begin{rem}
We point out that the same result (with the same proof) also holds for \emph{supersolutions} of problem~\eqref{eq-1}, i.~e. for $u \in C^{2,1}(M \times (0,+\infty)) \cap C^0 (M \times [0,+\infty))$ satisfying 
\begin{equation*}
u_t\geq \Delta u+ |u|^{p} V \quad \text{ in } M \times (0,+\infty) \quad \text{ and } \quad u (\cdot,0) = u_0\geq 0.
\end{equation*}
\end{rem}

\bibliographystyle{amsplain}
\bibliography{biblio}

\end{document}